\documentclass[vecarrow]{svmult}
\usepackage{makeidx}         
\usepackage{graphicx}        
\usepackage{psfrag}                             
\usepackage{multicol}        
\usepackage{amsmath}
\usepackage{amssymb}
\usepackage{amsfonts}
\usepackage{latexsym}
\usepackage[T1]{fontenc}
\usepackage{subfigure}                      
\usepackage{float}
\usepackage{mathrsfs}
\usepackage{mathptmx}         

\usepackage{color} 
\begin{document}

\bgroup


\newcommand{\rd}{\color{red}}

\newcommand{\rc}{\color{red}}

\newcommand{\RD}{\color{red}}

\newcommand{\bl}{\color{blue}}

\newcommand{\mg}{\color{magenta}}
\renewcommand{\mg}{}

\newcommand{\bn}{\color[rgb]{0.5,0.25,0.0}}

\newcommand{\gr}{\color[rgb]{0.0,0.45,0.0}}

\newcommand{\A}{{\cal A}}

\renewcommand{\D}{{\cal D}}

\renewcommand{\E}{{\cal E}}

\newcommand{\F}{{\cal F}}

\newcommand{\G}{{\cal G}}

\newcommand{\J}{{\cal J}}

\renewcommand{\L}{{\cal L}}

\newcommand{\Q}{{\cal Q}}

\newcommand{\C}{{\mathbb C}}
\newcommand{\R}{{\mathbb R}}
\newcommand{\N}{{\mathbb N}}

\newcommand{\T}{{\mathbb T}}
\newcommand{\Z}{\mathbb{Z}}

\newcommand{\Sc}{{\mathcal S}}

\newcommand{\sZ}{\sum_{\xi\in\Z^n}}
\renewcommand{\div}{{\rm div}}

\newcommand{\bs}{\boldsymbol}

\newcommand{\uxi}{\allmodesymb{\greeksym}{x}}

\newcommand{\mycomment}[1]{}

\allowdisplaybreaks

\title{Periodic Solutions in $\R^n$ for Stationary Anisotropic Stokes and Navier-Stokes Systems}
\index{stationary anisotropic Stokes and Navier-Stokes systems}

\titlerunning{Periodic Solutions for Stationary Anisotropic Stokes and Navier-Stokes}

\author{S. E. Mikhailov}

\authorrunning{S. E. Mikhailov}

\institute{S. E. Mikhailov\at Brunel University London, Uxbridge, UK,\hfill\break
\email{sergey.mikhailov@brunel.ac.uk}}

\maketitle

\index{periodic!solutions}
\index{periodic Sobolev spaces}
\index{solution regularity}

\abstract*{First, the solution uniqueness and existence of a stationary anisotropic (linear) Stokes system with constant viscosity coefficients in a compressible framework on $n$-dimensional flat torus are  analysed in a range of periodic Sobolev (Bessel-potential) spaces. 
By employing the Leray-Schauder fixed point theorem, the linear results are employed to show existence of solution to the stationary anisotropic (non-linear) Navier-Stokes incompressible system  on torus in a periodic Sobolev space.
Then the solution regularity results for stationary anisotropic Navier-Stokes system on torus are established.}

\section{Introduction}\label{sec:mik1}

Analysis of Stokes and Navier-Stokes equations is an established and active field of research in the applied mathematical analysis, see, e.g., \cite{Constantin-Foias1988, Galdi2011, RRS2016, Seregin2015, Sohr2001, Temam1995, Temam2001} and references therein.
In \cite{KMW2020, KMW-DCDS2021, KMW-LP2021, KMW-transv2021} this field has been extended to the transmission and boundary-value problems for stationary Stokes and Navier-Stokes equations of anisotropic fluids, particularly, with relaxed ellipticity condition on the viscosity tensor. In this chapter, we present some further results in this direction considering periodic solutions to the stationary Stokes and Navier-Stokes equations of anisotropic fluids, with an emphasis on solution regularity.  

First, the solution uniqueness and existence of a stationary, anisotropic (linear) Stokes system with constant viscosity coefficients in a compressible framework are analysed on $n$-dimensional flat torus in a range of periodic Sobolev (Bessel-potential) spaces. 
By employing the Leray-Schauder fixed point theorem, the linear results are {\mg used} to show existence of solution to the stationary anisotropic (non-linear) Navier-Stokes incompressible system  on torus in a periodic Sobolev space.
Then the solution regularity results for stationary anisotropic Navier-Stokes system on torus are established.

\section{Anisotropic Stokes and Navier-Stokes systems}\label{sec:mik2}

Let 
$\boldsymbol{\mathfrak L}$ denote a second order differential operator in the component-wise divergence form,
\begin{align*}
&(\boldsymbol{\mathfrak L}{\mathbf u})_k:=
\partial _\alpha\big(a_{kj}^{\alpha \beta }E_{j\beta }({\mathbf u})\big),\ \ k=1,\ldots ,n,
\end{align*}
were ${\mathbf u}\!=\!(u_1,\ldots ,u_n)^\top$, $E_{j\beta }({\mathbf u})\!:=\!\frac{1}{2}(\partial _ju_\beta +\partial _\beta u_j)$ are the entries of the symmetric part ${\mathbb E}({\mathbf u})$ of $\nabla {\mathbf u}$ (the gradient of ${\mathbf u}$),
and $a_{kj}^{\alpha \beta }$ are constant components of the tensor viscosity coefficient 
${\mathbb A}:=\!\left({a_{kj}^{\alpha \beta }}\right)_{1\leq i,j,\alpha ,\beta \leq n}$, cf. \cite{Duffy1978}.

Here and further on, the Einstein summation convention in repeated indices from $1$ to $n$ is used unless stated otherwise.

The following symmetry conditions are assumed (see \cite[(3.1),(3.3)]{Oleinik1992}),
\begin{align}
\label{eq:mik1}
a_{kj}^{\alpha \beta }=a_{\alpha j}^{{\mg k}\beta }=a_{k\beta }^{\alpha j}.
\end{align}

In addition, we require that tensor ${\mathbb A}$ satisfies the (relaxed) ellipticity condition  in terms of all {\it symmetric} matrices in ${\mathbb R}^{n\times n}$ with {\it zero matrix trace}, see \cite{KMW-DCDS2021,KMW-LP2021}. Thus, we assume that there exists a constant $C_{\mathbb A} >0$ such that, 
\begin{align}
\label{eq:mik2}
a_{kj}^{\alpha \beta }\zeta _{k\alpha }\zeta _{j\beta }\geq C_{\mathbb A}^{-1}|\pmb\zeta|^2\,,
\ \ &\forall\ \pmb\zeta =(\zeta _{k\alpha })_{k,\alpha =1,\ldots ,n}\in {\mathbb R}^{n\times n}\nonumber\\
&\mbox{ such that }\, \pmb\zeta=\pmb\zeta^\top \mbox{ and }
\sum_{k=1}^n\zeta _{kk}=0,
\end{align}
where $|\pmb\zeta |^2=\zeta _{k\alpha }\zeta _{k\alpha }$, and the superscript $\top $ denotes the transpose of a matrix.

The tensor ${\mathbb A}$  is endowed with the norm
\begin{align*}
\|{\mathbb A}\|:=\max\left\{|a_{kj}^{\alpha \beta }|:k,j,\alpha ,\beta =1\ldots ,n\right\}\,.
\end{align*}
Symmetry conditions \eqref{eq:mik1} lead to the following equivalent form of the operator $\boldsymbol{\mathfrak L}$
\begin{equation}
\label{eq:mik3}
\begin{array}{lll}
(\boldsymbol{\mathfrak L}{\mathbf u})_k=\partial _\alpha\big(a_{kj}^{\alpha \beta }\partial _\beta u_j\big),\ \ k=1,\ldots ,n.
\end{array}
\end{equation}

Let us also define the Stokes operator $\pmb{\mathcal L}$  as
\begin{align}
\label{eq:mik4}
\pmb{\mathcal L}({\mathbf u},p ):=\boldsymbol{\mathfrak L}{\mathbf u}-\nabla p.
\end{align}

Let ${\mathbf u}$ be an unknown vector field, $p $ be an unknown scalar field, ${\mathbf f}$ be a given vector field and $g$ be a given scalar field defined in $\T $. Then the equations
\begin{equation}
\label{eq:mik5}
\begin{array}{lll}
-\pmb{\mathcal L}({\mathbf u},p )={\mathbf f},\ {\rm{div}}\ {\mathbf u}=g \mbox{ in } \T
\end{array}
\end{equation}
determine the {\it anisotropic stationary Stokes system with viscosity tensor coefficient ${\mathbb A}=\left(A^{\alpha \beta }\right)_{1\leq \alpha ,\beta \leq n}$ in a compressible framework}.

In addition, the following nonlinear system
\begin{align}
\label{eq:mik6}
-\pmb{\mathcal L}({\mathbf u},p )+{({\mathbf u}\cdot \nabla ){\mathbf u}}={\mathbf f}\,, \ \ {\rm{div}} \, {\mathbf u}=g  \mbox{ in } \T 
\end{align}
is called the {\it anisotropic stationary Navier-Stokes system with viscosity tensor coefficient ${\mathbb A}\!=\!\left(A^{\alpha \beta }\right)_{1\leq \alpha ,\beta \leq n}$ in a compressible framework}.
If $g=0$ in \eqref{eq:mik5} and \eqref{eq:mik6}, then these equations are reduced, respectively, to the {\it incompressible anisotropic stationary Stokes and Navier-Stokes systems}.

In the {\it isotropic case}, the tensor ${\mathbb A}$ reduces to
\begin{align}
\label{eq:mik7}
a_{kj}^{\alpha \beta}=\lambda \delta _{k\alpha }\delta _{j\beta }+\mu \left(\delta_{\alpha j}\delta _{\beta k}+\delta_{\alpha \beta }\delta _{kj}\right),\ 1\leq i,j, \alpha ,\beta \leq n\,,
\end{align}
where $\lambda$ and $\mu$ are real constant parameters with
$\mu>0$ (cf., e.g., Appendix III, Part I, Section 1 in \cite{Temam2001}),
and \eqref{eq:mik3} becomes
\begin{equation}
\label{eq:mik8}
\boldsymbol{\mathfrak L}{\mathbf u}
=(\lambda +\mu)\nabla{\rm div}\,\mathbf u+\mu\Delta \mathbf u.
\end{equation}
Then it is immediate that condition \eqref{eq:mik2} is fulfilled (cf. \cite{KMW-LP2021}) and thus our results apply also to the Stokes and Navier-Stokes systems in the {\it isotropic case}. Assuming $\lambda=0$, $\mu=1$ we arrive at the classical mathematical formulations of isotropic Stokes and Navier-Stokes systems.

\section{Some function spaces on torus}\label{sec:mik3}
Let us introduce some function spaces on torus and periodic function spaces (see, e.g., 
\cite[p.26]{Agmon1965},  \cite{Agranovich2015}, \cite{McLean1991}, \cite[Chapter 3]{RT-book2010}, \cite[Section 1.7.1]{RRS2016}  
\cite[Chapter 2]{Temam1995}, 
for more details).

Let $n\ge 1$  be an integer and $\T$ be the $n$-dimensional torus that can be parametrized as the semi-open cube $\T= [0,1)^n\subset\R^n$, cf.  \cite[p. 312]{Zigmund2002}.
In what follows, ${\mathcal D}(\T)=\mathcal C^\infty(\T)$ denotes the space of infinitely smooth real or complex functions on the torus.
As usual, $\N$ denotes the set of natural numbers,  $\N_0$ the set of natural numbers complemented by 0, and $\mathbb{Z}$ the set of integers.

Let   $\pmb\xi \in \mathbb{Z}^n$ denote the $n$-dimensional vector with integer components. 
We will further need also the set 
$$\dot\Z^n:=\Z^n\setminus\{\mathbf 0\}.$$
Extending the torus parametrisation to $\R^n$, it is often useful to identify $\T$ with the quotient space $\R^n\setminus \Z^n$. 
Then the space of functions $\mathcal C^\infty(\T)$ on the torus can be identified with the space of $\T$-periodic (1-periodic) functions 
$\mathcal C^\infty_\#=\mathcal C^\infty_\#(\R^n)$ that consists of functions $\phi\in \mathcal C^\infty(\R^n)$ such that
\begin{align*}
\phi(\mathbf x+\pmb\xi)=\phi(\mathbf x)\quad \forall\,  \pmb\xi \in \mathbb{Z}^n.
\end{align*}
Similarly, the Lebesgue space on the torus $L_{p}(\T)$, $1\le p\le\infty$,  can be identified with the periodic Lebesgue space $L_{p\#}=L_{p\#}(\R^n)$ that consists of functions $\phi\in L_{p,\rm loc}(\R^n)$, which satisfy the periodicity condition \eqref{eq:mik9} for a.e. $\mathbf x$.

The space dual  to $\mathcal D(\T)$, i.e.,  the space of linear bounded functionals on $\mathcal D(\T)$,  called the space of torus distributions is denoted by $\mathcal D'(\T)$ and can be identified with the space of periodic distributions $\mathcal D'_\#$ acting {\mg on} $\mathcal C^\infty_\#$.

The toroidal/periodic Fourier transform 
mapping a  function $g\in \mathcal C_\#^\infty$ to a set of its Fourier coefficients $\hat g$ is defined as (see, e.g., \cite[Definition 3.1.8]{RT-book2010})  
\begin{align*}
 \hat g(\pmb\xi)=[\F_{\T} g](\pmb\xi):=\int_{\T}e^{-2\pi i \mathbf x\cdot\pmb\xi}g(\mathbf x)d\mathbf x,\quad \pmb\xi\in\Z^n.
\end{align*}
and can be generalised to  the Fourier transform acting on distribution a $g\in\mathcal D'_\#$.

For any $\pmb\xi\in\Z^n$, let $|\pmb\xi|:=(\sum_{j=1}^n \xi_j^2)^{-1/2}$ be the Euclidean norm in $\Z^n$ and let us denote $$\rho(\pmb\xi):=(1+|\pmb\xi|^2)^{1/2}.$$
Evidently,
\begin{align}\label{eq:mik9}
\frac{1}{2}\rho(\pmb\xi)^2\le |\pmb\xi|^2\le \rho(\pmb\xi)^2\quad\forall\,\pmb\xi\in \dot\Z^n.
\end{align}

Similar to \cite[Definition 3.2.2]{RT-book2010}, for $s\in\R$ we define the {\em periodic/toroidal Sobolev (Bessel-potential) spaces} $H_\#^s:=H_\#^s(\R^n):=H^s(\T)$, which consist of the torus distributions $ g\in\mathcal D'(\T)$, for which the norm
\begin{align}\label{eq:mik10}
\| g\|_{H_\#^s}:=\| \rho^s\widehat g\|_{\ell_2}:=\left(\sZ\rho(\pmb\xi)^{2s}|\widehat g(\pmb\xi)|^2\right)^{1/2}
\end{align}
is finite, i.e., the series in \eqref{eq:mik10} converges.
Here $\| \cdot\|_{\ell_2}$ is the standard norm in the space of square summable sequences.
By \cite[Proposition 3.2.6]{RT-book2010}, $H_\#^s$ are Hilbert spaces.

For $g\in H_\#^s$, $s\in\R$, and  $m\in\N_0$, let us consider the partial sums 
$$g_m(\mathbf x)=\sum_{\pmb\xi\in\Z^n, |\pmb\xi|\le m}\hat g(\pmb\xi)e^{2\pi i \mathbf x\cdot\pmb\xi}.$$ 
 Evidently, $g_m\in \mathcal C_\#^\infty$, $\hat g_m(\pmb\xi)=\hat g(\pmb\xi)$ if $|\pmb\xi|\le m$ and  $\hat g_m(\pmb\xi)=0$ if $|\pmb\xi|> m$.
This implies that $\|g-g_m\|_{H_\#^s}\to 0$ as $m\to\infty$ and hence we can write 
\begin{align}\label{eq:mik11}
g(\mathbf x)=\sum_{\pmb\xi\in\Z^n}\hat g(\pmb\xi)e^{2\pi i \mathbf x\cdot\pmb\xi},
\end{align}
where the Fourier series converges in the sense of norm \eqref{eq:mik10}.
Moreover, since $g$ is an arbitrary distribution from $H_\#^s$, this also implies that the space ${\mathcal C}^\infty_\#$ is dense in $H_\#^s$ for any $s\in\R$ (cf. \cite[Exercise 3.2.9]{RT-book2010}).

There holds the compact embedding $H_\#^t\hookrightarrow H_\#^s$ if $t>s$,  embeddings $H_\#^s\subset \mathcal C_\#^m$ if $m\in\N_0$, $s>m+n/2$, and moreover, $\bigcap_{s\in\R}H_\#^s={\mathcal C}^\infty_\#$ (cf. \cite[Exercises 3.2.10, 3.2.10 and Corollary 3.2.11]{RT-book2010}). Note also that the torus norms on $H_\#^s$ are equivalent to the corresponding standard (non-periodic) Bessel potential norms on $\T$ as a cubic domain, see, e.g., \cite[Section 13.8.1]{Agranovich2015}.

By \eqref{eq:mik10}, 
$\| g\|^2_{H_\#^s}=|\widehat g(\mathbf 0)|^2 +| g|^2_{H_\#^s},$ 
where
\begin{align*}
| g|_{H_\#^s}:=\| \rho^s\widehat g\|_{\dot\ell_2}:=\left(\sum_{\pmb\xi\in\dot\Z^n}\rho(\pmb\xi)^{2s}|\widehat g(\pmb\xi)|^2\right)^{1/2}
\end{align*}
is the seminorm in $H_\#^s$.

For any $s\in\R$, let us also introduce  the space
$
\dot H_\#^s:=\{g\in H_\#^s: \langle g,1\rangle_{\T}=0\}.
$
The definition implies that if $g\in \dot H_\#^s$, then $\widehat g(\mathbf 0)=0$ and 
\begin{align}\label{eq:mik12}
\| g\|_{\dot H_\#^s}=\| g\|_{H_\#^s}=| g|_{H_\#^s}=\| \rho^s\widehat g\|_{\dot\ell_2}\ .
\end{align}
Denoting 
$
\dot {\mathcal C}^\infty_\#:=\{g\in {\mathcal C}^\infty_\#: \langle g,1\rangle_{\T}=0\}
$,
then $\bigcap_{s\in\R}\dot H_\#^s=\dot {\mathcal C}^\infty_\#$.

The corresponding spaces of $n$-component vector functions/distribution are denoted as $\mathbf H_\#^s:=(H_\#^s)^n$, etc.

Note that  the norm $\|\nabla (\cdot )\|_{{\mathbf H}_\#^{0}}$
is an equivalent norm in $\dot H_\#^1$. 
Indeed, by \eqref{eq:mik11}
\begin{align*}
\nabla g(\mathbf x)=2\pi i\sum_{\pmb\xi\in\dot\Z^n}\pmb\xi e^{2\pi i \mathbf x\cdot\pmb\xi}\hat g(\pmb\xi),\quad
\widehat{\nabla g}(\pmb\xi)=2\pi i\pmb\xi \hat g(\pmb\xi)
\end{align*}
and 
then \eqref{eq:mik9} and \eqref{eq:mik12} imply
\begin{multline}\label{eq:mik13}
{\mg 2\pi^2\|g\|^2_{H_\#^1}=}2\pi^2\|g\|^2_{\dot H_\#^1}=2\pi^2|g|^2_{H_\#^1}\le \| \nabla g\|^2_{{\mathbf H}_\#^0}\\
\le 4\pi^2|g|^2_{H_\#^1}={\mg 4\pi^2\|g\|^2_{\dot H_\#^1}=}4\pi^2\| g\|^2_{H_\#^1} \quad \forall\,g\in {\mg\dot H_\#^1}.
\end{multline}
The vector counterpart of \eqref{eq:mik13} takes form
\begin{align}\label{eq:mik14}
2\pi^2\| \mathbf v\|^2_{{\mathbf H}_\#^1}{\mg=2\pi^2\| \mathbf v\|^2_{\dot{\mathbf H}_\#^1}}
\le \| \nabla \mathbf v\|^2_{(H_\#^0)^{n\times n}}
\le {\mg 4\pi^2\| \mathbf v\|^2_{\dot{\mathbf H}_\#^1}=}4\pi^2\| \mathbf v\|^2_{{\mathbf H}_\#^1} \quad \forall\,\mathbf v\in \dot {\mathbf H}_\#^1.
\end{align}

We will further need also the first Korn inequality
\begin{align}
\label{eq:mik15}
\|\nabla {\bf v}\|^2_{(L_{2\#})^{n\times n}}\leq 2\|\mathbb E ({\bf v})\|^2_{(L_{2\#})^{n\times n}}\quad\forall\, \mathbf v\in {\mathbf H}_\#^1 
\end{align}
that can be easily proved by adapting, e.g., the proof in \cite[Theorem 10.1]{McLean2000} {\mg to} the periodic Sobolev space.

Let us define the Sobolev spaces of divergence-free functions/distributions,
\begin{align*}
\dot{\mathbf H}_{\#\sigma}^{s}
&:=\left\{{\bf w}\in\dot{\mathbf H}_\#^{s}:{\div}\, {\bf w}=0\right\},\quad s\in\R,
\end{align*}
endowed with the same norm \eqref{eq:mik10}.

\section{Stationary anisotropic Stokes system on {\mg flat} torus}\label{sec:mik4}
In this section, we generalise to the isotropic and anisotropic (linear) Stokes systems in compressible framework and to a range of Sobolev spaces the analysis, available in \cite[Section 2.2]{Temam1995} 

For the unknowns $({\mathbf u},p )\in \dot{\mathbf H}_\#^s\times \dot H_\#^{s-1}$ and the given data
$({\mathbf f},g)\in\dot{\mathbf H}_\#^{s-2}\times \dot H_\#^{s-1}$, $s\in\R$, let us consider the Stokes system
\begin{align}
\label{eq:mik16}
-\pmb{\mathcal L}({\mathbf u},p )&=\mathbf{f},\\
\label{eq:mik17}
{\rm{div}}\, {\mathbf u}&=g,
\end{align}
that should be understood in the sense of distributions, i.e.,
\begin{align}
\label{eq:mik18}
-\langle\pmb{\mathcal L}({\mathbf u},p ), \pmb\phi\rangle_{\T}&=\langle\mathbf{f}, \pmb\phi\rangle_{\T}\quad\forall\,\pmb\phi\in ({\mathcal C}^\infty_\#)^n,\\
\label{eq:mik19}
\langle{\rm{div}}\, {\mathbf u}, \phi\rangle_{\T}&=\langle g, \phi\rangle_{\T}\quad\forall\,\phi\in {\mathcal C}^\infty_\#.
\end{align}
For $\dot\Z^n$, let us employ  $\bar e_{\pmb\xi}(\mathbf x)=e^{-2\pi i x\cdot\pmb\xi}$ as $\phi$ in \eqref{eq:mik19} and $\bar e_{\pmb\xi}(\mathbf x)$, multiplied by the coordinate vector, as $\pmb\phi$ in \eqref{eq:mik18}. 
Then recalling \eqref{eq:mik3}-\eqref{eq:mik4}, we arrive for each $\pmb\xi\in \dot\Z^n$ at the following algebraic system for the Fourier coefficients, $\hat u_j(\pmb\xi)$, $k=1,2,\dots,n$, and $\hat p(\pmb\xi)$.
\begin{align}
\label{eq:mik20}
4\pi^2\xi_\alpha a_{kj}^{\alpha \beta}\xi_\beta \hat u_j(\pmb\xi) + 2\pi i\xi_k\hat p(\pmb\xi)&=\hat{f}_k(\pmb\xi) \quad\forall\,\pmb\xi\in \dot\Z^n,\  k=1,2,\dots,n\\
\label{eq:mik21}
2\pi i\xi_j\hat{u}_j(\pmb\xi)&=\hat g(\pmb\xi) \quad\ \forall\,\pmb\xi\in \dot\Z^n.
\end{align}

The $(n+1) \times (n+1)$ matrix, $\mathfrak S(\pmb\xi)$, of system \eqref{eq:mik20}-\eqref{eq:mik21} is in fact the principal symbol of the anisotropic Stokes system \eqref{eq:mik16}-\eqref{eq:mik17} that was analysed in \cite[Lemma 15]{KMW-LP2021} to prove that the Stokes system is elliptic in the sense of Agmon–Douglis–Nirenberg. 
It was, particularly proved that the matrix $\mathfrak S$ is non-singular if $\pmb\xi\ne 0$ and hence the solution of system \eqref{eq:mik20}-\eqref{eq:mik21} can be represented in terms of the inverse matrix $\mathfrak S^{-1}(\pmb\xi)$ as 
\begin{align}\label{eq:mik22}
\left(\widehat{\mathbf u}(\pmb\xi) \atop \hat p(\pmb\xi)\right)=\mathfrak S^{-1}(\pmb\xi)\left(\widehat{\mathbf f}(\pmb\xi) \atop \hat g(\pmb\xi)\right) \quad\forall\,\pmb\xi\in \dot\Z^n.
\end{align}
Moreover, using the estimates for the matrix, obtained in that lemma proof, and implementing to the algebraic system the variant of Babuska-Brezzi theory given in Theorem 2.34 and Remark 2.35(i) in \cite{Ern-Guermond2004}, see also \cite[Theorem 10]{KMW-LP2021},
we obtain the following estimates for the solution of the algebraic system \eqref{eq:mik20}-\eqref{eq:mik21},
\begin{align}
\label{eq:mik23}
&|\widehat{\mathbf u}(\pmb\xi)|\le C_{uf}\frac{|\widehat{\mathbf f}(\pmb\xi)|}{|2\pi\pmb\xi|^2}
+C_{ug}\frac{|\hat{g}(\pmb\xi)|}{2\pi|\pmb\xi|},\\
\label{eq:mik24}
&
|\hat p(\pmb\xi)|\le C_{pf}\frac{|\widehat{\mathbf f}(\pmb\xi)|}{2\pi|\pmb\xi|}
+C_{pg}|\hat{g}(\pmb\xi)|\quad\forall\,\pmb\xi\in \dot\Z^n,
\end{align}
where\quad
$
 C_{uf}=2C_{\mathbb A},\ 
C_{ug}=C_{pf}=1+2C_{\mathbb A}\|\mathbb A\|,\ 
C_{pg}=\|\mathbb A\|(1+2C_{\mathbb A}\|\mathbb A\|).
$

\begin{remark}\label{rem:mik1}
For the isotropic case \eqref{eq:mik7}, 
due to \eqref{eq:mik8},
system \eqref{eq:mik20}-\eqref{eq:mik21} reduces to
\begin{align}
\label{eq:mik25}
&4\pi^2\left[(\lambda+\mu) \pmb\xi(\pmb\xi\cdot \widehat{\mathbf u}(\pmb\xi))+\mu|\pmb\xi|^2\widehat{\mathbf u}(\pmb\xi)\right]
+2\pi i\pmb\xi\hat p(\pmb\xi)=\widehat{\mathbf f}(\pmb\xi), \quad\forall\,\pmb\xi\in \dot\Z^n,\\
\label{eq:mik26}
&2\pi i\pmb\xi\cdot \widehat{\mathbf u}(\pmb\xi)=\hat g(\pmb\xi) 
\quad\forall\,\pmb\xi\in \dot\Z^n.
\end{align}
Taking scalar product of equation \eqref{eq:mik25} with $\pmb\xi$ and employing \eqref{eq:mik26}, we obtain
\begin{align}
\label{eq:mik27}
&\hat p(\pmb\xi)=\frac{\pmb\xi\cdot\widehat{\mathbf f}(\pmb\xi)}{2\pi i|\pmb\xi|^2} +(\lambda+2\mu) {\hat g(\pmb\xi)}, \quad\forall\,\pmb\xi\in \dot\Z^n,
\end{align}
and substituting this back to \eqref{eq:mik25}, we get
\begin{align}
\label{eq:mik28}
& 
\widehat{\mathbf u}(\pmb\xi)
=\frac{1}{4\pi^2\mu|\pmb\xi|^2}\left[\widehat{\mathbf f}(\pmb\xi) 
-\pmb\xi\frac{\pmb\xi\cdot\widehat{\mathbf f}(\pmb\xi)}{|\pmb\xi|^2}\right]
+ \pmb\xi\frac{\hat g(\pmb\xi)}{2\pi i|\pmb\xi|^2},
\quad\forall\,\pmb\xi\in \dot\Z^n
\end{align}
(cf. \cite[Section 2.2]{Temam1995} for the case $s=1$, $g=0$, $\lambda=0$, and $\mu=1$).
Expressions  \eqref{eq:mik28}, \eqref{eq:mik27}  evidently satisfy estimates \eqref{eq:mik23}, \eqref{eq:mik24}.
\hfill$\square$
\end{remark}

The anisotropic Stokes system \eqref{eq:mik16}-\eqref{eq:mik17} can be re-written as  
\begin{align*}
&S\left({\mathbf u}\atop p\right)=\left({\mathbf f} \atop g\right),
\end{align*}
where
\begin{align*}
&S\left({\mathbf u}\atop p\right):=\left(-\pmb{\mathcal L}({\mathbf u},p) \atop {\rm{div}}\, {\mathbf u}\right),
\end{align*}
and for any $s\in\R$,
\begin{align}\label{eq:mik29}
{S}:\dot{\mathbf H}_\#^s\times \dot H_\#^{s-1} \to \dot{\mathbf H}_\#^{s-2}\times \dot H_\#^{s-1} 
\end{align}
is a linear continuous operator.

Now we are in the position to prove the following assertion.
\begin{theorem}\label{th:mik1}
Let condition \eqref{eq:mik2} hold.

(i)
For any
$({\mathbf f},g)\in\dot{\mathbf H}_\#^{s-2}\times \dot H_\#^{s-1}$, $s\in\R$,
the anisotropic Stokes system \eqref{eq:mik16}-\eqref{eq:mik17} in torus $\T$ has a unique solution
$({\mathbf u},p )\in \dot{\mathbf H}_\#^s\times \dot H_\#^{s-1} $, where
\begin{align}\label{eq:mik30}
{\mathbf u}(\mathbf x)=\sum_{\pmb\xi\in\dot\Z^n}e^{2\pi i x\cdot\pmb\xi}\widehat {\mathbf u}(\pmb\xi),\quad
p(\mathbf x)=\sum_{\pmb\xi\in\dot\Z^n}e^{2\pi i x\cdot\pmb\xi}\hat p(\pmb\xi)
\end{align}
with $\widehat {\mathbf u}(\pmb\xi)$ and $\hat p(\pmb\xi)$ given by \eqref{eq:mik22}. 
In addition, there exists a constant $C=C(C_{\mathbb A},n)>0$ such that
\begin{align}
\label{eq:mik31}
\|{\mathbf u}\|_{\mg\dot{\mathbf H}_\#^{s}}+\|p \|_{\mg\dot H_\#^{s-1} }
\leq C\left(\|\mathbf f \|_{\dot{\mathbf H}_\#^{s-2}}+\|g\|_{\mg\dot H_\#^{s-1} }\right)
\end{align}
and operator \eqref{eq:mik29} is an isomorphism.

(ii) Moreover, if $({\mathbf f},g)\in(\dot{\mathcal C}_\#^\infty)^n\times \dot {\mathcal C}^\infty_\#$
then $({\mathbf u},p)\in(\dot{\mathcal C}_\#^\infty)^n\times \dot {\mathcal C}^\infty_\#$.
\end{theorem}
\begin{proof}
(i) Expressions \eqref{eq:mik22} supplemented by the relations $\widehat{\mathbf u}(\mathbf 0)=\mathbf 0$, $\hat p(\mathbf 0)=0$ imply the uniqueness. 
From estimates \eqref{eq:mik23} 
and \eqref{eq:mik24} 
we obtain the estimate
\begin{align*}
\| {\mathbf u}\|_{\mg\dot{\mathbf H}_\#^{s}}&=\left(\sum_{\pmb\xi\in\dot\Z^n}\rho(\pmb\xi)^{2s}|\widehat {\mathbf u}(\pmb\xi)|^2\right)^{1/2}\nonumber\\
&\le \frac{C_{uf}}{4\pi^2}\left(\sum_{\pmb\xi\in\dot\Z^n}\rho(\pmb\xi)^{2s}
\frac{|\widehat{\mathbf f}(\pmb\xi)|^2}{|\pmb\xi|^4}\right)^{1/2}
+\frac{C_{ug}}{2\pi}\left(\sum_{\pmb\xi\in\dot\Z^n} \rho(\pmb\xi)^{2s} \frac{|\hat{g}(\pmb\xi)|^2}{|\pmb\xi|^2}\right)^{1/2}\nonumber\\
&= \frac{C_{uf}}{4\pi^2}\left(\sum_{\pmb\xi\in\dot\Z^n}\rho(\pmb\xi)^{2(s-2)}|\widehat{\mathbf f}(\pmb\xi)|^2
\frac{\rho(\pmb\xi)^4}{|\pmb\xi|^4}\right)^{1/2}\nonumber\\
&\quad+\frac{C_{ug}}{2\pi}\left(\sum_{\pmb\xi\in\dot\Z^n}\rho(\pmb\xi)^{2(s-1)}|\hat{g}(\pmb\xi)|^2
\frac{\rho(\pmb\xi)^2}{|\pmb\xi|^2}\right)^{1/2}\nonumber\\
&\le \frac{C_{uf}}{2\pi^2}\|\mathbf f \|_{\mg\dot{\mathbf H}_\#^{s-2}}
+\frac{C_{ug}}{2\pi}\sqrt{2}\|g\|_{\mg\dot H_\#^{s-1}}
\end{align*}
and the similar estimate for $\|p \|_{\mg\dot H_\#^{s-1}}$, which imply \eqref{eq:mik31} and hence inclusions in the corresponding spaces.

(ii) The inclusion $({\mathbf f},g)\in(\dot{\mathcal C}^\infty_\#)^n\times \dot {\mathcal C}^\infty_\#$ implies that $({\mathbf f},g)\in\dot{\mathbf H}_\#^{s-2}\times \dot H_\#^{s-1}$ for any $s\in\R$.
Then  by item (i) $({\mathbf u},p )\in \dot{\mathbf H}_\#^s\times \dot H_\#^{s-1} $ for any $s\in\R$  and hence
$({\mathbf u},p)\in(\dot{\mathcal C}^\infty_\#)^n\times \dot {\mathcal C}^\infty_\#$.
\hfill$\square$
\end{proof}

If $g=0$ in \eqref{eq:mik17}, we can re-formulate the Stokes system \eqref{eq:mik16}-\eqref{eq:mik17} as one vector equation
\begin{align}
\label{eq:mik32}
&-\pmb{\mathcal L}({\mathbf u},p )=\mathbf{f}
\end{align}
for the unknowns $({\mathbf u},p )\in \dot{\mathbf H}_{\#\sigma}^{s}\times \dot H_\#^{s-1}$ and the given data
${\mathbf f}\in\dot{\mathbf H}_\#^{s-2}$, $s\in\R$.
Then Theorem \ref{th:mik1} implies the following assertion.
\begin{corollary}\label{cor:mik1}
Let condition \eqref{eq:mik2} hold.

(i)
For any
${\mathbf f}\in\dot{\mathbf H}_\#^{s-2}$, $s\in\R$,
the anisotropic Stokes equation \eqref{eq:mik32} in torus $\T$ has a unique incompressible  solution
$({\mathbf u},p )\in \dot{\mathbf H}_{\#\sigma}^{s}\times \dot H_\#^{s-1} $, 
with $\widehat {\mathbf u}(\pmb\xi)$ and $\hat p(\pmb\xi)$ given by \eqref{eq:mik22}, \eqref{eq:mik30} (and particularly by \eqref{eq:mik28}, \eqref{eq:mik27}, \eqref{eq:mik30} for the isotropic case \eqref{eq:mik7})  with $g=0$.
In addition, there exists a constant $C=C(C_{\mathbb A},n)>0$ such that
\begin{align*}
\|{\mathbf u}\|_{\mg\dot{\mathbf H}_\#^{s}}+\|p \|_{\mg\dot H_\#^{s-1}}
\leq C\|\mathbf f \|_{\dot{\mathbf H}_\#^{s-2}}
\end{align*}
and the operator 
$$\pmb{\mathcal L}: \dot{\mathbf H}_{\#\sigma}^{s}\times \dot H_\#^{s-1} \to \dot{\mathbf H}_\#^{s-2}$$
 is an isomorphism.

(ii) Moreover, if ${\mathbf f}\in(\dot{\mathcal C}_\#^\infty)^n$
then $({\mathbf u},p)\in(\dot{\mathcal C}_\#^\infty)^n\times \dot {\mathcal C}^\infty_\#$.
\end{corollary}

\section{Stationary anisotropic Navier-Stokes system with constant coefficients on torus}\label{sec:mik5}

\subsection{Existence of a weak solution to anisotropic incompressible Navier-Stokes system on torus}

In this section, we show the existence of a weak solution of  the anisotropic Navier-Stokes system in the incompressible case with general data in $L^2$-based Sobolev spaces on a torus ${\T}$, for $n\in\{2,3\}$.
We use the well-posedness result established in Theorem \ref{th:mik1} for the Stokes system on a torus  and  the  following variant of the {\it Leray-Schauder fixed point theorem} (see, e.g., \cite[Theorem 11.3]{Gilbarg-Trudinger}).
\begin{theorem}
\label{th:mik2}
Let ${B}$ denote a Banach space and $T:{B}\to {B}$ be a continuous
and compact operator. If there exists a constant $M_0>0$ such that {$\|x\|_{B}\leq M_0$} for every pair $(\mathbf x,\theta)\in {B}\times [0,1]$ satisfying $\mathbf x=\theta T\mathbf x$, then the operator $T$ has a fixed point 
$\mathbf x_0$ $($with $\|x_0\|_B\leq M_0$$)$.
\end{theorem}

Let us consider the Navier-Stokes system
\begin{align}
\label{eq:mik33}
-&\pmb{\mathcal L}({\mathbf u},p )=\mathbf{f}-({\mathbf u}\cdot \nabla ){\mathbf u},\\
\label{eq:mik34}
&{\rm{div}}\, {\mathbf u}=0,
\end{align}
for the couple of unknowns $({\mathbf u},p )\in {\mg\dot{\mathbf H}}_{\#}^{1}\times \dot H_\#^0 $ and the given data
$\mathbf f\in {\mg\dot{\mathbf H}}_{\#}^{-1} $.
As for the Stokes system, the Navier-Stokes system \eqref{eq:mik33}- \eqref{eq:mik34} can be re-written as one vector equation
\begin{align}
\label{eq:mik35}
&-\pmb{\mathcal L}({\mathbf u},p )=\mathbf{f}-({\mathbf u}\cdot \nabla ){\mathbf u}
\end{align}
for the  unknowns $({\mathbf u},p )\in \dot{\mathbf H}_{\#\sigma}^{1}\times \dot H_\#^0 $ and the given data
$\mathbf f\in {\mg\dot{\mathbf H}}_{\#}^{-1} $.

Let us denote the nonlinear operator as $\pmb{\mathit B}$, i.e., 
\begin{align}
\label{eq:mik36}
\pmb{\mathit B} {\mathbf w}:&=({\mathbf w}\cdot \nabla ){\mathbf w},\ \ \forall \, {\mathbf w}\in {\mathbf H}_{\#}^{s}, \ s\in\R.
\end{align}

\begin{theorem}\label{th:mik3}
Let the operator $\pmb{\mathit B}: {\mathbf w}\mapsto \pmb{\mathit B}\mathbf w$ be defined by \eqref{eq:mik36} and let $n\ge 2$. 

(i) If \, $0<s<n/2$
then 
\begin{align}
\label{eq:mik37}
&{\pmb{\mathit B}}:\dot{\mathbf H}_{\#\sigma}^{s}\to \dot{\mathbf H}_{\#}^{2s-1-n/2}
\end{align}
is a well defined, continuous and bounded quadratic operator, i.e., there exists $C_{n,s}>0$ such that
\begin{align}
\label{eq:mik38}
&\left\|\pmb{\mathit B} {\mathbf w}\right\|_{{\mathbf H}_{\#}^{2s-1-n/2}}
\le  C_{n,s}\|{\mathbf w}\|^2_{{\mathbf H}_{\#}^{s}} 
\quad\forall\ {\mathbf w}\in {\mathbf H}_{\#}^{s}.
\end{align}

(ii) If $s>n/2$ then  
\begin{align}\label{eq:mik39}
&{\pmb{\mathit B}}:\dot{\mathbf H}_{\#\sigma}^{s}\to \dot{\mathbf H}_{\#}^{s-1}
\end{align}
is well defined, continuous and bounded quadratic operator, i.e., there exists $C_{n,s}>0$ such that
\begin{align}
\label{eq:mik40}
&\left\|\pmb{\mathit B} {\mathbf w}\right\|_{{\mathbf H}_\#^{s-1}}
\le  C_{n,s}\|{\mathbf w}\|^2_{{\mathbf H}_{\#}^{s}} 
\quad\forall\ {\mathbf w}\in {\mathbf H}_{\#}^{s}.
\end{align}
\end{theorem}
\begin{proof}
If a function $\mathbf w$ is periodic, then evidently the function  $\pmb{\mathit B}\mathbf w$ is periodic as well.

(i) Let  $0<s<n/2$.
Due to Theorem 1(iii) in Section 4.6.1 of \cite{Runst-Sickel1996} 
and equivalence of the Bessel potential norms on square and norms \eqref{eq:mik10} for the Sobolev spaces on torus, we have, 
\begin{align}
\label{eq:mik41}
&\left\|({\mathbf v}_1\cdot \nabla ){\mathbf v}_2\right\|_{{\mathbf H}_{\#}^{2s-1-n/2}}
\le  C_{n,s}\|{\mathbf v}_1\|_{{\mathbf H}_{\#}^{s}} \|{\mathbf v}_2\|_{{\mathbf H}_{\#}^{s}},
\quad\forall\ {\mathbf v}_1,{\mathbf v}_2\in {\mathbf H}_{\#}^{s}.
\end{align}
for some constant  $C_{n,s}>0$. 
This particularly implies estimate \eqref{eq:mik38}.

Further, if ${\mathbf u}\in \dot{\mathbf H}_{\#\sigma}^{s}$ then 
$$
\langle \pmb{\mathit B}\mathbf u,1\rangle_{\T}
=\langle {\mathbf u}\cdot \nabla ){\mathbf u} ,1\rangle_{\T}
=-\langle ({\rm div}\, {\mathbf u}){\mathbf u},1 \rangle_{\T}
=\mathbf 0
$$
since ${\rm div}\, {\mathbf u}=0$.
Together with estimate \eqref{eq:mik38} this implies that quadratic operator \eqref{eq:mik37} is well defined and bounded.

Let ${\mathbf w},{\mathbf w}'\in \dot{\mathbf H}_{\#\sigma}^{1}$.
Then by  \eqref{eq:mik41} we obtain
\begin{align*}
\big\|\pmb{\mathit B}{\mathbf w}-\pmb{\mathit B}{\mathbf w}'\big\|_{{\mathbf H}_{\#}^{2s-1-n/2}}
&\le\left\|({\mathbf w}\cdot \nabla ){\mathbf w}-({\mathbf w}'\cdot \nabla ){\mathbf w}'\right\|_{{\mathbf H}_{\#}^{2s-1-n/2}}\nonumber\\
&\le\left\|(({\mathbf w}-{\mathbf w}')\cdot \nabla ){\mathbf w} 
+ ({\mathbf w}'\cdot \nabla )({\mathbf w}-{\mathbf w}')\right\|_{{\mathbf H}_{\#}^{2s-1-n/2}}\nonumber\\
&\le  C_{n,s}\left\|{\mathbf w}-{\mathbf w}'\right\|_{{\mathbf H}_{\#}^{s}}\left( \|{\mathbf w}\|_{{\mathbf H}_{\#}^{s}}
+ \|{\mathbf w}'\|_{{\mathbf H}_{\#}^{s}}
\right).
\end{align*}
This estimate shows that operator \eqref{eq:mik37} is continuous. 

(ii) Let  $s>n/2$.
Due to Theorem 1(i) in Section 4.6.1 of \cite{Runst-Sickel1996} 
and equivalence of the Bessel potential norms and norms \eqref{eq:mik10} for the Sobolev spaces on torus, we have, 
\begin{align*}
&\left\|({\mathbf v}_1\cdot \nabla ){\mathbf v}_2\right\|_{{\mathbf H}_\#^{s-1}}
\le  C_{n,s}\|{\mathbf v}_1\|_{{\mathbf H}_{\#}^{s}} \|{\mathbf v}_2\|_{{\mathbf H}_{\#}^{s}},
\quad\forall\ {\mathbf v}_1,{\mathbf v}_2\in {\mathbf H}_{\#}^{s}.
\end{align*}
for some constant  $C_{n,s}>0$. 
This particularly implies estimate \eqref{eq:mik40} and then the boundedness of operator \eqref{eq:mik39}. By the same arguments as in item (i), one can prove that this operator is also well defined and continuous.
\hfill$\square$
\end{proof}
\begin{corollary}\label{Fcomp}
Let $n\in \{2,3\}$. 
Then the quadratic operator
\begin{align}\label{eq:mik42}
&{\pmb{\mathit B}}:\dot{\mathbf H}_{\#\sigma}^{1}\to {\mathbf H}_{\#}^{-1}
\end{align}
is well defined, continuous, bounded and compact.
\end{corollary}
\begin{proof}
Let $n=3$.
Due to Theorem~\ref{th:mik3}(i), the operator
$
{\pmb{\mathit B}}:\dot{\mathbf H}_{\#\sigma}^{1}\to \dot{\mathbf H}_{\#}^{-1/2}
$
is well defined, continuous and bounded.
On the other hand, the compactness of  embedding $H_\#^{-1/2}\hookrightarrow H_\#^{-1}$ implies  the compactness of  embedding $\dot H_\#^{-1/2}\hookrightarrow \dot H_\#^{-1}$ and hence
gives  the compactness of operator \eqref{eq:mik42} and thus the corollary claim for $n=3$.

Let now $n=2$. Then by Theorem~\ref{th:mik3}(i), the operator 
$
{\pmb{\mathit B}}:\dot{\mathbf H}_{\#\sigma}^{s}\to \dot{\mathbf H}_{\#}^{2s-2}
$
is well defined, continuous and bounded for any $s\in(1/2,1)$.
In addition, for $s\in(1/2,1)$ we also have the compact embeddings $\dot H_{\#\sigma}^{1}\hookrightarrow \dot H_{\#\sigma}^{s}$ and   $\dot H_\#^{2s-2}\hookrightarrow \dot H_\#^{-1}$ that lead to the corollary claim for $n=2$.
\hfill$\square$ 
\end{proof}

Next we show the existence of a weak solution of the Navier-Stokes equation. 
\begin{theorem}
\label{th:mik4}
Let $n\in \{2,3\}$ and condition \eqref{eq:mik2} hold.
If \,$\mathbf f\in {\mg\dot{\mathbf H}}_{\#}^{-1} $, then
the anisotropic Navier-Stokes equation \eqref{eq:mik35}
has a solution $({\mathbf u},p)\in \dot{\mathbf H}_{\#\sigma}^{1} \times \dot H_\#^0 $.
\end{theorem}
\begin{proof}

We will reduce the analysis of the nonlinear equation \eqref{eq:mik35}
to the analysis of a nonlinear operator in the Hilbert space $\dot{\mathbf H}_{\#\sigma}^{1}$ and show that this operator has a fixed-point due to the Leray-Schauder Theorem.

Nonlinear equation \eqref{eq:mik35} can be re-written as 
\begin{align}
\label{eq:mik43}
&-\pmb{\mathcal L}({\mathbf u},p )=\mathbf{f}-\pmb{\mathit B}{\mathbf u}.
\end{align}

By Corollary \ref{cor:mik1}, the linear operator
\begin{align}\label{eq:mik44}
-\pmb{\mathcal L}:\dot{\mathbf H}_{\#\sigma}^{1}\times \dot H_\#^0 \to \dot{\mathbf H}_{\#}^{-1}
\end{align}
is an isomorphism.
Its inverse operator, $-\pmb{\mathcal L}^{-1}$, can be split into two operator components,
$$-\pmb{\mathcal L}^{-1}=\left(\pmb{\mathcal U}\atop\mathcal P\right),$$
where
$\pmb{\mathcal U}: {\mg\dot{\mathbf H}}_{\#}^{-1} \to \dot{\mathbf H}_{\#\sigma}^{1}$
and $\mathcal P:{\mg\dot{\mathbf H}}_{\#}^{-1} \to \dot H_\#^0 $ are linear continuous operators such that
$$-\pmb{\mathcal L}\left(\pmb{\mathcal U}\pmb{\mathcal F}\atop\mathcal P\pmb{\mathcal F}\right)=\pmb{\mathcal F}$$ 
for any $\pmb{\mathcal F}\in \dot{\mathbf H}_{\#}^{-1}$.
Applying the inverse operator, $-\pmb{\mathcal L}^{-1}$, to equation \eqref{eq:mik43}, we reduce it to the equivalent nonlinear system
\begin{align}
\label{eq:mik45}
{\mathbf u}
&={\mathbf U}{\mathbf u},\\
\label{eq:mik46}
p&= P{\mathbf u},
\end{align}
where ${\mathbf U}:\dot{\mathbf H}_{\#\sigma}^{1}\to \dot{\mathbf H}_{\#\sigma}^{1}$
and
$P:\dot{\mathbf H}_{\#\sigma}^{1}\to \dot H_\#^0 $
are the nonlinear operators defined as
\begin{align}
\label{eq:mik47}
&{\mathbf U}{\mathbf w}:=\pmb{\mathcal U}(\mathbf{f}-\pmb{\mathit B}{\mathbf w}),\\
\label{eq:mik48}
& P{\mathbf w}:=\mathcal P(\mathbf{f}-\pmb{\mathit B}{\mathbf w})
\end{align}
for the fixed $\mathbf f$.

Since $p$ is not involved in \eqref{eq:mik45}, we will first prove the existence of a solution ${\mathbf u}\in \dot{\mathbf H}_{\#\sigma}^{1}$ to this equation. Then we use \eqref{eq:mik46} as a representation formula for $p$, which gives the existence of the pressure field $p\in \dot H_\#^0 $.
In order to show the existence of a fixed point of the operator ${\mathbf U}$ and, thus, the existence of a solution of equation \eqref{eq:mik45},
we employ Theorem \ref{th:mik2}.

By Corollary~\ref{Fcomp}{\mg, for $n\in \{2,3\}$} the operator
$\pmb{\mathit B}:\dot{\mathbf H}_{\#\sigma}^{1}\to {\mathbf H}_{\#}^{-1}$
is bounded, continuous and compact.
Since $\mathbf f\in {\mathbf H}_{\#}^{-1}$ is fixed and the operator 
$\pmb{\mathcal U}: {\mathbf H}_{\#}^{-1} \to \dot{\mathbf H}_{\#\sigma}^{1}$
is linear and continuos, definition \eqref{eq:mik47} implies that the operator
${\mathbf U}:\dot{\mathbf H}_{\#\sigma}^{1}\to  \dot{\mathbf H}_{\#\sigma}^{1}$ 
is also bounded, continuous, and compact.

Next, we show that there exists a constant $M_0>0$ such that if ${\mathbf w}\in \dot{\mathbf H}_{\#\sigma}^{1}$ satisfies the equation
\begin{align}
\label{eq:mik49}
&{\mathbf w}=\theta{\mathbf U}{\mathbf w}
\end{align}
for some $\theta \in [0,1]$, then $\|{\mathbf w}\|_{\dot{\mathbf H}_{\#\sigma}^{1}}\le M_0$.
Let us denote
\begin{align}\label{eq:mik50}
q:=\theta P{\mathbf w}.
\end{align}
By applying the operator $-\pmb{\mathcal L}$ to equations \eqref{eq:mik49}-\eqref{eq:mik50} and by using relations \eqref{eq:mik47} and \eqref{eq:mik48}, we deduce that whenever the pair
$({\mathbf w},\theta)\in \dot{\mathbf H}_{\#\sigma}^{1}\times\R$ satisfies equation \eqref{eq:mik49}, then the equation
\begin{align*}
&-\pmb{\mathcal L}({\mathbf w},q)=\theta (\mathbf{f}-\pmb{\mathit B}{\mathbf w}),
\end{align*}
is also satisfied due to the isomorphism property of operator \eqref{eq:mik44}.
This equation should be understood in the sense of distribution, i.e.,
\begin{multline}
\label{eq:mik51}
\langle-\pmb{\mathcal L}({\mathbf w},q), \pmb\phi\rangle_{\T}
=\left\langle a_{ij}^{\alpha \beta }E_{j\beta }({\mathbf w}),E_{i\alpha }({\pmb\phi})\right\rangle _{\T }
-\langle {\mg q},{\rm{div}}\, {\pmb\phi} \rangle _{\T}\\
=\theta\langle\mathbf{f}-\pmb{\mathit B}{\mathbf w}, \pmb\phi\rangle_{\T}\quad\forall\,\pmb\phi\in ({\mathcal C}^\infty_\#)^n.
\end{multline}
Taking into account that the space $({\mathcal C}^\infty_\#)^n$ is dense in ${\mathbf H}_{\#}^{1}$ and the continuity of the dual products in \eqref{eq:mik51} with respect to $\pmb\phi\in{\mathbf H}_{\#}^{1}$ , equation \eqref{eq:mik51} should hold also for $\pmb\phi=\mathbf w\in \dot{\mathbf H}_{\#\sigma}^{1}$.
Then we obtain
\begin{align}
\label{eq:mik52}
&\left\langle a_{ij}^{\alpha \beta }E_{j\beta }({\mathbf w}),E_{i\alpha }({\mathbf w})\right\rangle _{\T }
=\theta \langle\mathbf{f}-\pmb{\mathit B}{\mathbf w},{\mathbf w}\rangle_{\T}.
\end{align}

Since ${\mathbf w}\in \dot{\mathbf H}_{\#\sigma}^{1}$, relation \eqref{eq:mik55} implies that
$ \langle\pmb{\mathit B}{\mathbf w},{\mathbf w}\rangle_{\T}=\langle({\mathbf w}\cdot \nabla ){\mathbf w},{\mathbf w}\rangle _{\T }=0$.
Then by using the norm equivalence \eqref{eq:mik14},  the Korn first inequality \eqref{eq:mik15}, the ellipticity condition \eqref{eq:mik2}, equation \eqref{eq:mik52},
and the H\"older inequality, we obtain for $\theta\ge 0$ that

\begin{align*}
\| \mathbf w\|^2_{\mg\dot{\mathbf H}_{\#}^{1}}
&\le\frac{1}{2\pi^2}\|\nabla {\mathbf w}\|^2_{(L_{2\#})^{\mg n\times n}}
\leq\frac{1}{\pi^2}\|{\mathbb E}({\mathbf w})\|_{(L_{2\#} )^{n\times n}}^2\nonumber\\
&
\leq \frac{1}{\pi^2}C_{\mathbb A}\left\langle a_{ij}^{\alpha \beta }E_{j\beta }({\mathbf w}),E_{i\alpha }({\mathbf w})
\right\rangle _{\T } 
\leq \frac{\theta}{\pi^2} C_{\mathbb A} \|{\mathbf f}\|_{\mg\dot{\mathbf H}_{\#}^{-1}} \|{\mathbf w}\|_{\mg\dot{\mathbf H}_{\#}^{1}}.
\end{align*}
Hence, for $\theta\in[0,1]$,
\begin{align*}
\| \mathbf w\|_{\mg\dot{\mathbf H}_{\#}^{1}}
&\le M_0:=\frac{1}{\pi^2}C_{\mathbb A} \|{\mathbf f}\|_{\mg\dot{\mathbf H}_{\#}^{-1}}.
\end{align*}

Therefore, the operator $\mathbf U :\dot{\mathbf H}_{\#\sigma}^{1}\to  \dot{\mathbf H}_{\#\sigma}^{1}$ satisfies the hypothesis of Theorem \ref{th:mik2} (with $B=\dot{\mathbf H}_{\#\sigma}^{1}$), and hence it has a fixed point ${\mathbf u}\!\in \!\dot{\mathbf H}_{\#\sigma}^{1}$, that is, ${\mathbf u}=\mathbf U {\mathbf u}$. Then with $p\in \dot H_\#^0 $ as in \eqref{eq:mik46}, we obtain that the couple $(\mathbf u, p)\in \dot{\mathbf H}_{\#\sigma}^{1}\times \dot H_\#^0 $ satisfies the nonlinear equation \eqref{eq:mik35}.
\hfill$\square$
\end{proof}

\subsection{Solution regularity for the stationary anisotropic Navier-Stokes system}

In this section, using the bootstrap argument we show that the regularity of a solution of  the anisotropic incompressible Navier-Stokes system on ${\T}^n$, $n\ge\mg 2$ is completely determined by the regularity of its right-hand side, as for the Stokes system.
To prove this we use the inclusions of the nonlinear term $\pmb{\mathit B} {\mathbf u}$ given by Theorem~\ref{th:mik3} and the unique solvability of corresponding (linear) Stokes system.

\begin{theorem}\label{th:mik5}
Let condition \eqref{eq:mik2} hold.
Let $n\ge 2$ and  ${\mg -1+n/2}<s_1<s_2$.

(i) If $({\mathbf u},p)\in \dot{\mathbf H}_{\#\sigma}^{s_1} \times \dot H_\#^{s_1-1} $ is a solution of the anisotropic Navier-Stokes equation \eqref{eq:mik35} with a right hand side $\mathbf f\in \dot{\mathbf H}_{\#}^{s_2-2} $,
then  $({\mathbf u},p)\in \dot{\mathbf H}_{\#\sigma}^{s_2} \times \dot H_\#^{s_2-1} $.

(ii) Moreover, if ${\mathbf f}\in(\dot{\mathcal C}_\#^\infty)^n$
then $({\mathbf u},p)\in(\dot{\mathcal C}_\#^\infty)^n\times \dot {\mathcal C}^\infty_\#$.
\end{theorem}
\begin{proof}
(i) Let $({\mathbf u},p)\in \dot{\mathbf H}_{\#\sigma}^{s_1} \times \dot H_\#^{s_1 -1} $ be a solution of  \eqref{eq:mik35} with  $\mathbf f\in \dot{\mathbf H}_{\#}^{s_2-2} $.
Then by Theorem~\ref{th:mik3}, for the nonlinear term we have the inclusion $\pmb{\mathit B}\mathbf u\in \dot{\mathbf H}_\#^{t_1}$ with $t_1=2s_1-1-n/2$ if $s_1<n/2$, with $t_1=s_1-1$ if $s_1>n/2$, and with any $t_1\in (s_1-2,s_1-1)$ (and we can further use $t_1=s_1-3/2$ for certainty) if $s_1=n/2$. 
Hence the couple $({\mathbf u},p)$ satisfies the equation
\begin{align}
\label{eq:mik53}
&-\pmb{\mathcal L}({\mathbf u},p )=\mathbf{f}^{(1)}
\end{align}
with $\mathbf f^{(1)}:=\mathbf f -\pmb{\mathit B}\mathbf u\in \dot{\mathbf H}_\#^{s^{(1)}-2}$, where $s^{(1)}=\min\{s_2, t_1+2\}$.
By Corollary~\ref{cor:mik1}(i), the linear equation \eqref{eq:mik53} has a unique solution in 
$\dot{\mathbf H}_{\#\sigma}^{s}\times \dot H_\#^{s-1}$ for any $s\le s^{(1)}$ and thus 
$({\mathbf u},p)\in \dot{\mathbf H}_{\#\sigma}^{s^{(1)}}\times \dot H_\#^{s^{(1)}-1}$.
If $s^{(1)}=s_2$, which we call {\mg case} (a),  this proves item (i) of the theorem.

Otherwise we have {\mg case} (b), when $s^{(1)}<s_2$,  i.e., $s^{(1)}=t_1+2$,
by the definition of  $\mg s^{(1)}$.
Then  we arrange an iterative process by replacing $s_1$ with $s^{(1)}=t_1+2$ on each iteration until we arrive at {\mg case} (a), thus proving item (i) of the theorem.
Note that in {\mg case} (b),
\begin{align*}
s^{(1)}-s_1\ge \delta:=\min\{s_1+1-n/2, 1, 1/2\}>0
\end{align*} 
in the first iteration, and $\delta$ can only increase in the next iterations due to the increase of $s_1$.
This implies that the iteration process will reach the case (a) and stop after a finite number of iterations. 

(ii) If ${\mathbf f}\in(\dot{\mathcal C}_\#^\infty)^n$, then for any $s_2\in\R$ we have $\mathbf f\in \dot{\mathbf H}_{\#}^{s_2-2} $  and item (i) implies that $({\mathbf u},p)\in \dot{\mathbf H}_{\#\sigma}^{s_2} \times \dot H_\#^{s_2-1} $. 
Hence $({\mathbf u},p)\in(\dot{\mathcal C}_\#^\infty)^n\times \dot {\mathcal C}^\infty_\#$.
\hfill$\square$
\end{proof}

Combining Theorems \ref{th:mik4} and \ref{th:mik5}, we obtain the following assertion on existence and regularity of solution to the Navier-Stokes system on torus.
\begin{theorem}
\label{th:mik6}
Let $n\in \{2,3\}$ and condition \eqref{eq:mik2} hold.

(i) If $\mathbf f\in \dot{\mathbf H}_{\#}^{s-2} $, $s\ge 1$, then the anisotropic Navier-Stokes equation \eqref{eq:mik35} has a solution 
$({\mathbf u},p)\in \dot{\mathbf H}_{\#\sigma}^{s} \times \dot H_\#^{s-1} $.

(ii) Moreover, if ${\mathbf f}\in(\dot{\mathcal C}_\#^\infty)^n$ then equation \eqref{eq:mik35} has a solution 
$({\mathbf u},p)\in(\dot{\mathcal C}_\#^\infty)^n\times \dot {\mathcal C}^\infty_\#$.
\end{theorem} 

Note that in the {\em isotropic case} \eqref{eq:mik7} with $\lambda=0$, similar results for the Navier-Stokes in torus as well as in domains of $\R^n$  are available, e.g., in \cite{Galdi2011, RRS2016, Seregin2015, Sohr2001,  Temam2001}.

\section{Some auxiliary results}\label{sec:mik6}

The dense embedding of the space $({\mathcal C}^\infty_\#)^n$ into ${\mathbf H}_{\#}^{1}$ and the divergence theorem 
imply the following identity for any ${\mathbf v}_1,{\mathbf v}_2, {\mathbf v}_3\in {\mathbf H}_{\#}^{1}$.
\begin{align}
\label{eq:mik54}
\left\langle({\mathbf v}_1\cdot \nabla ){\mathbf v}_2,{\mathbf v}_3\right\rangle _{\T}
&=\int_{\T}\nabla\cdot\left({\mathbf v}_1({\mathbf v}_2\cdot {\mathbf v}_3)\right)d{\mathbf x}
-\left\langle(\nabla \cdot{\mathbf v}_1){\mathbf v}_3
+({\mathbf v}_1\cdot \nabla ){\mathbf v}_3,{\mathbf v}_2\right\rangle _{\T}\nonumber
\\
&=
-\left\langle({\mathbf v}_1\cdot \nabla ){\mathbf v}_3,{\mathbf v}_2\right\rangle _{\T}
-\left\langle(\nabla \cdot{\mathbf v}_1){\mathbf v}_3,{\mathbf v}_2\right\rangle _{\T}.
 \end{align}
In view of \eqref{eq:mik54} we obtain the identity
\begin{align*}
\left\langle({\mathbf v}_1\cdot \nabla ){\mathbf v}_2,{\mathbf v}_3\right\rangle _{\mg\T}
&\!\!=\!-\left\langle({\mathbf v}_1\cdot \nabla ){\mathbf v}_3,{\mathbf v}_2\right\rangle _{\mg\T}\quad
 \forall\ {\mathbf v}_1\in {\mathbf H}_{\#\sigma}^{1},\
{\mathbf v}_2,\, {\mathbf v}_3\in {\mathbf H}_{\#}^{1}\,,
\end{align*}
and hence the well known formula
\begin{equation}
\label{eq:mik55}
\left\langle ({\mathbf v}_1\cdot \nabla ){\mathbf v}_2,{\mathbf v}_2\right\rangle _{\mg\T}=0\quad \forall \,
{\mathbf v}_1\!\in \!{\mathbf H}_{\#\sigma}^{1},\ {\mathbf v}_2\in {\mathbf H}_{\#}^{1}.
\end{equation}

\par\egroup
\end{document}